\documentclass{IEEEtran}
\usepackage{cite}
\usepackage{amsmath,amssymb,amsfonts,amsthm}
\usepackage{algorithmic}
\usepackage{graphicx}
\usepackage{textcomp}

\newcommand{\ts}[1]{\textsuperscript{#1}}
\newcommand{\ds}[1]{\displaystyle{#1}}
\newcommand{\bs}[1]{\boldsymbol{#1}}

\newcommand{\wt}[1]{\widetilde{#1}}
\newcommand{\mb}[1]{\mathbb{#1}}
\newcommand{\da}[1]{\stackrel{\leftrightarrow}{#1}}

\newtheorem{lemma}{Lemma}
\newtheorem{theorem}{Theorem}

\def\BibTeX{{\rm B\kern-.05em{\sc i\kern-.025em b}\kern-.08em
    T\kern-.1667em\lower.7ex\hbox{E}\kern-.125emX}}
    
\begin{document}

\title{Tight Bounds on Polynomials and Its Application to Dynamic Optimization Problems}

\author{Eduardo M. G. Vila, Eric C. Kerrigan, \IEEEmembership{Senior Member, IEEE}, and Paul Bruce
\thanks{ 
This work was supported by the Science and Technology Facilities Council (STFC) under a doctoral training grant (ST/V506722/1). For the purpose of open access, the authors have applied a creative commons attribution (CC BY) licence  to any author accepted manuscript version arising. This study does not involve any underlying research data. All results are derived from theoretical analysis and are fully described within the article.}
\thanks{E. M. G. Vila is with the Department of Electrical and Electronic Engineering, Imperial College London, SW7 2AZ, UK (e-mail: {\tt\small eduardo.vila16@alumni.imperial.ac.uk}).}
\thanks{E. C. Kerrigan is with the Department of Electrical and Electronic Engineering and the Department of Aeronautics, Imperial College London, SW7 2AZ, UK (e-mail: {\tt\small e.kerrigan@imperial.ac.uk}).}
\thanks{P. Bruce is with the Department of Aeronautics, Imperial College London, SW7 2AZ, UK (e-mail: {\tt\small p.bruce@imperial.ac.uk}).}
}

\maketitle

\begin{abstract}
This paper presents a pseudo-spectral method for Dynamic Optimization Problems (DOPs) that allows for tight polynomial bounds to be achieved via flexible sub-intervals. 
The proposed method not only rigorously enforces inequality constraints, but also allows for a lower cost in comparison with non-flexible discretizations. 
Two examples are provided to demonstrate the feasibility of the proposed method to solve optimal control problems.
Solutions to the example problems exhibited up to a tenfold reduction in relative cost.
\end{abstract}


\section{Introduction}
\label{sec:introduction}

\subsection{Constraining Polynomials}

Polynomials are attractive for approximating functions because any continuous function on a bounded interval can be approximated with arbitrary accuracy by a polynomial \cite[Ch.~6]{trefethen_approximation_2019}. This type of approximation offers a rapid convergence rate under certain smoothness conditions \cite[Ch.~7]{trefethen_approximation_2019}. Together with numerical considerations, these theoretical advantages make polynomials the favored option for function approximation. 

A severe limitation of using polynomials arises in the presence of constraints.
It is not trivial to ensure that a polynomial $p$ satisfies some lower and upper bounds, $p_\ell$ and $p_u$, respectively, throughout a finite interval $[t_a, t_b]$, i.e., that
\begin{equation}
    p_\ell \leq p(t) \leq p_u
    \quad
    \forall t \in [t_a, t_b].
\label{eq:poly}
\end{equation}
Often in practice, only a finite number of sample points of~$p$ are constrained, with no guarantee of constraint satisfaction between the samples.


To rigorously enforce \eqref{eq:poly}, one may represent $p$ as a sum-of-squares (SOS) to ensure $p$ is non-negative.
SOS conditions can guarantee non-negativity over the entire domain, or bounded interval \cite[Sect.~1.21]{szego_orthogonal_1939}. 
In practice, SOS conditions are formulated as semi-definite constraints, which require specialized solution techniques.
The lack of compatibility with nonlinear optimization methods renders the SOS technique unappealing for many problems.

\begin{figure}
    \centering
    \includegraphics[width=\linewidth]{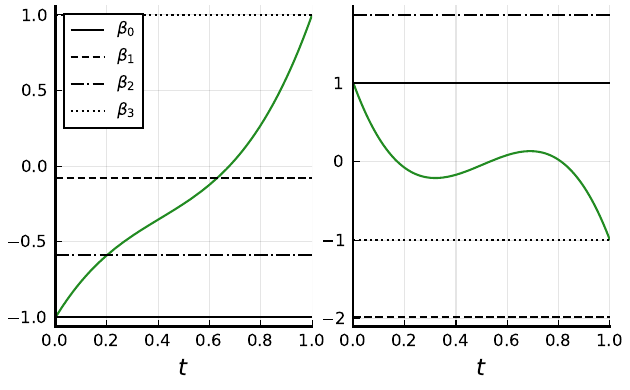}
    \vspace*{-9mm}\caption{Examples of 3\ts{rd} degree polynomials and their Bernstein coefficients~$\beta_j$.}
    \label{fig:hull_examples}
\end{figure}

A promising approach is to express $p$ in the Bernstein polynomial basis (detailed in Section~\ref{sec:bernstein}).
Under this basis, the polynomial is guaranteed to be inside the convex hull of its coefficients, in the interval $[0, 1]$ \cite{cargo_bernstein_1966}.
Thus, \eqref{eq:poly} can be satisfied by directly constraining the Bernstein coefficients of~$p$ to lie between $p_\ell$ and $p_u$.
Figure~\ref{fig:hull_examples} shows two examples of polynomials with their Bernstein coefficients.

\subsection{Achieving Tight Bounds}

In general, the bounds provided by the Bernstein coefficients are not always tight, i.e., they are potentially conservative approximations of the exact minimum and maximum values of $p$ in the interval~$[0, 1]$.
The impact of this conservatism is illustrated by the following simple approximation problem.

\begin{figure}
    \centering
    \includegraphics[width=\linewidth]{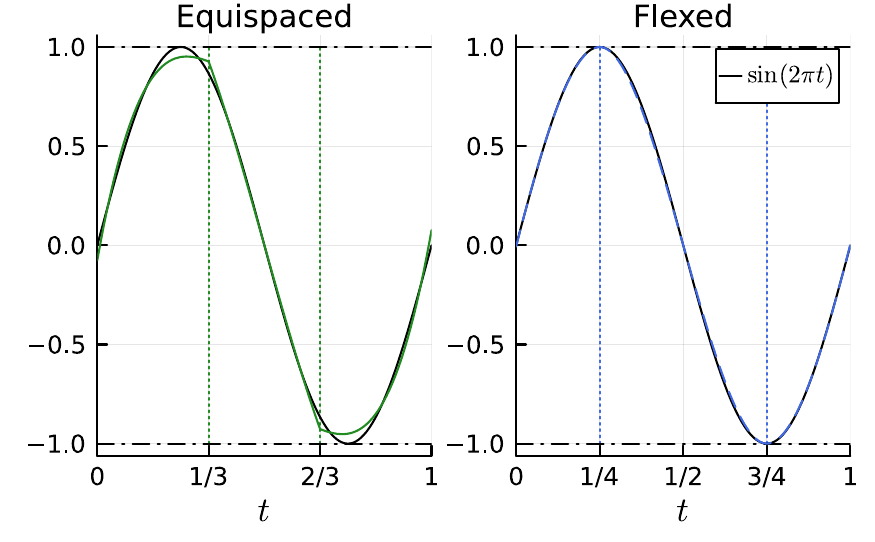}
    \vspace*{-9mm}\caption{Bernstein-constrained piecewise polynomial approximations with equispaced and flexed sub-intervals.}
    \label{fig:sine_solutions}
\end{figure}

Consider using polynomials to approximate the function
\begin{equation}
    t \mapsto \sin(2 \pi t),
    \quad
    t \in [0, 1],
\label{eq:sine}
\end{equation}
such that the approximation is constrained between -1 and~1. 
Piecewise polynomials are used with two different partitions of $[0, 1]$: a partition with equispaced sub-intervals and a partition with \emph{flexed} (i.e., not equispaced) sub-intervals.
To satisfy the constraints, the Bernstein coefficients of the polynomials are constrained between -1 and 1 (the remaining details can be found in the Appendix).
Figure~\ref{fig:sine_solutions} shows how the bounds may or may not be tight, depending on how the interval $[0, 1]$ is partitioned.
With equispaced sub-intervals, the Bernstein bounds on the polynomials are not tight, resulting in a significant approximation error near the constraints.
With an appropriate \emph{flexing} of the sub-intervals, the bounds become tight, and the approximation error is greatly reduced.

\begin{figure}
    \centering
    \includegraphics[width=\linewidth]{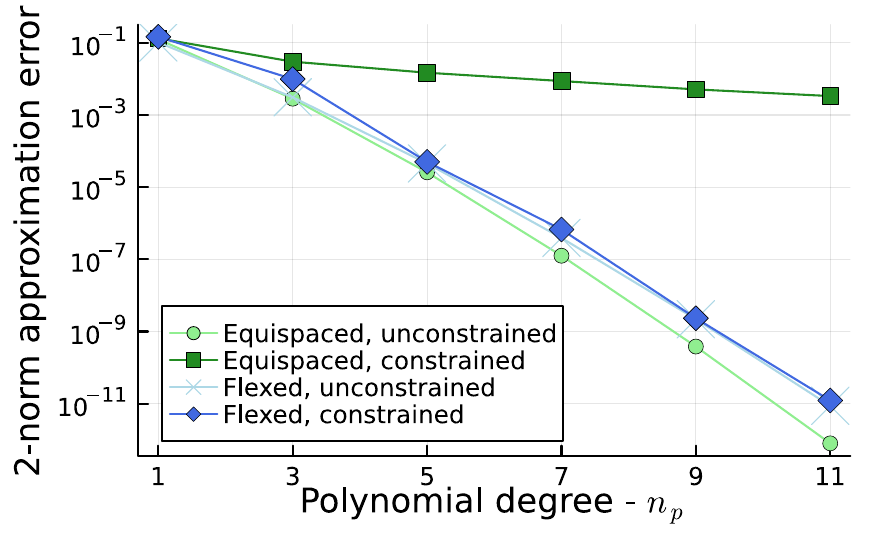}
    \vspace*{-9mm}\caption{Approximation errors for equispaced and flexed sub-intervals, with and without Bernstein constraints.}
    \label{fig:sine_convergence}
\end{figure}

An even larger impact is seen in the rate of convergence of the approximation, for increasing polynomial degrees.
Figure~\ref{fig:sine_convergence} shows how conservative Bernstein bounds can impact the rate of convergence of the approximation error.
With equispaced sub-intervals, the constraints hinder the otherwise sharp rate of convergence, whereas with flexed sub-intervals, the unconstrained rate of convergence is mostly preserved. 

\subsection{Dynamic Optimization}

In this article, the ability to tightly constrain polynomials is developed, with an emphasis on computing solutions to Dynamic Optimization Problems (DOPs).
By combining the properties of Bernstein polynomials with sub-interval flexibility, the proposed method is able to provide rigorous constraint satisfaction, often without compromising the optimality of the solution to the DOP.

DOPs are concerned with finding state and input trajectories that minimize a given cost function, subject to a variety of constraints.
Sub-classes of DOPs include optimal control, also known as trajectory optimization, state estimation and system identification problems, as well as initial value and boundary value problems of differential (algebraic) equations.

Due to their high convergence rate, polynomial-based methods are the state of the art for numerically solving DOPs~\cite{ross_review_2012}.  
By using orthogonal polynomials, these pseudo-spectral methods can obtain an exponential rate of convergence, known as the spectral rate.
Recently, a long-standing ill-conditioning issue of pseudo-spectral methods has been resolved by using Birkhoff interpolation, in lieu of the traditional Lagrange interpolation \cite{koeppen_fast_2019}.
For the most part however, pseudo-spectral methods do not rigorously enforce polynomial constraints as in \eqref{eq:poly}, instead only constraining samples of the trajectory values.

Recently, a non-pseudo-spectral method purely based on Bernstein polynomials has been proposed for DOPs \cite{cichella_optimal_2021}.
Despite the advantageous constraint satisfaction properties, the method results in a slower rate of convergence.
Later, a pseudo-spectral method that uses Bernstein constraints was proposed \cite{allamaa_safety_2023}, capable of attaining a spectral rate of convergence.
This method was extended to use (fixed) sub-intervals \cite{allamaa_real-time_2024}, so as to reduce (but not eliminate) the conservatism of the Bernstein constraints. 

\subsection{Contributions and Outline}

In this article, we propose a pseudo-spectral DOP method that uses \emph{flexible} sub-intervals, which may lead to tightly constrained polynomials.
Previously, the use of flexible sub-intervals was driven by their ability to represent discontinuities in the solutions, both in collocation methods \cite{ross_pseudospectral_2004} and in integrated-residual methods \cite{nita_fast_2022}. The proposed method thus inherits this property.

The contributions of this article are as follows:
\begin{itemize}
    \item We show that monotonic polynomials are not necessarily tightly bounded by their Bernstein coefficients, via an example. 
    \item We prove that a finite number of sub-intervals is sufficient to tightly bound (piecewise) monotonic polynomials.
    \item We propose a method for obtaining numerical solutions to DOPs, where a flexible discretization is used to promote tight bounds on constrained dynamic variables represented by polynomials.
\end{itemize}

Section~\ref{sec:bernstein} introduces the Bernstein polynomial basis and presents theoretical results on tight polynomial bounds.
Section~\ref{sec:problem} defines a general DOP formulation, together with common equivalent definitions.
Section~\ref{sec:discretization} describes the flexible DOP discretization.
Section~\ref{sec:example} demonstrates the proposed method on two example DOPs.
Finally, Section~\ref{sec:conclusion} provides concluding remarks and directions for further research.

\section{Tightly Bounded Polynomials}
\label{sec:bernstein}

\subsection{The Bernstein Basis}

Let $p : [0, 1] \rightarrow \mb R$ be a polynomial of degree at most $n_p$.
One may write $p$ as
\begin{equation}
    p(t)
    = 
    \alpha_0 + \alpha_1 t + ... + \alpha_{n_p} t^{n_p}
    =
    \sum_{k=0}^{n_p} \alpha_k t^k,
\end{equation}
where $\{\alpha_k\}_{k=0}^{n_p}$ are the \emph{monomial} coefficients of $p$.
Under this representation, there is little hope of using the coefficients to provide bounds on $p$, let alone obtaining tight bounds.
To that effect, a strategic change of polynomial basis is used.

In the Bernstein polynomial basis, we represent $p$ as
\begin{equation}
    p(t)
    =
    \sum_{j=0}^{n_p} \beta_j b_{n_p,j}(t),
\label{eq:bases}
\end{equation}
where $\{\beta_j\}_{j=0}^{n_p}$ are the \emph{Bernstein} coefficients and $\{ b_{n_pj} \}_{j=0}^{n_p}$ are the Bernstein basis polynomials. The latter are given by
\begin{equation}
    b_{n_p,j}(t) := \binom{n_p}{j} t^j (1-t)^{(n_p-j)},
\label{eq:bernstein_basis}
\end{equation}
which are depicted in Figure~\ref{fig:basis} for the case of $n_p=4$.
\begin{figure}
    \centering
    \includegraphics[width=\linewidth]{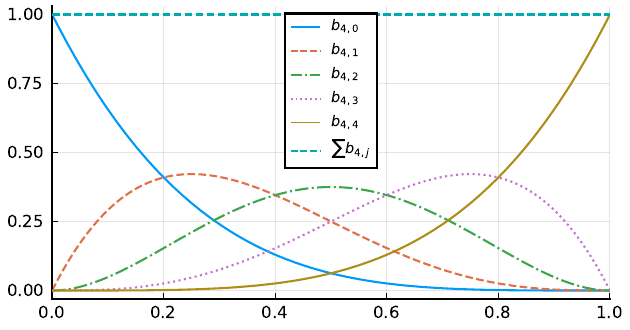}
    \vspace*{-9mm}\caption{Bernstein basis polynomials of degree 4, along with their sum.}
    \label{fig:basis}
\end{figure}

This change of basis can be performed on the coefficients of $p$ using the relationship\cite{cargo_bernstein_1966}
\begin{equation}
    \beta_j = \sum_{k=0}^j \alpha_k \left. \binom{j}{k} \middle/ \binom{n_p}{k} \right. .
    \label{eq:bernstein_coefficients}
\end{equation}
In matrix form, we write
\begin{equation}
    \beta = B \alpha,
\end{equation}
where $\beta$ and $\alpha$ are ordered column vectors of $\{ \beta_j\}_{j=0}^{n_p}$ and $\{ \alpha_k\}_{k=0}^{n_p}$, respectively, and each element \( B_{j,k} \) of the lower triangular matrix \( B \) is defined as:
\begin{equation}
B_{j,k} = 
\begin{cases}
\left. \binom{j}{k} \middle/ \binom{n_p}{k} \right. & \text{if}~ k \leq j, \\
0 & \text{otherwise}.
\end{cases}
\end{equation}

The following known result \cite{cargo_bernstein_1966} states how the Bernstein coefficients $\{\beta_j\}_{j=0}^{n_p}$ can provide bounds for $p$.
\begin{lemma}
    Let $p_{\min}$ and $p_{\max}$ be the minimum and maximum values of $p(t)$~$\forall t \in [0, 1]$.
    In $[0, 1]$, $p$ is bounded by
    \begin{equation}
        \min \{\beta_j\}_{j=0}^{n_p} \leq p_{\min} 
        \quad \text{and} \quad
        \max \{\beta_j\}_{j=0}^{n_p} \geq p_{\max}.
    \label{eq:hull}
    \end{equation}
    Moreover, the bounds are tight in the following cases: 
    \begin{align}
        \beta_0 &= \min \{\beta_j\}_{j=0}^{n_p}
        &&\iff&
        \beta_0 &= p_{\min},\label{eq:min_increasing}\\
        \beta_{n_p} &= \max \{\beta_j\}_{j=0}^{n_p}
        &&\iff&
        \beta_{n_p} &= p_{\max},\label{eq:max_increasing}\\
        \beta_0 &= \max \{\beta_j\}_{j=0}^{n_p}
        &&\iff&
        \beta_0 &= p_{\max},\label{eq:max_decreasing}\\
        \beta_{n_p} &= \min \{\beta_j\}_{j=0}^{n_p}
        &&\iff&
        \beta_{n_p} &= p_{\min}.\label{eq:min_decreasing}
    \end{align}
\label{th:bounds}
\end{lemma}
In other words, Lemma~\ref{th:bounds} states that a polynomial is bounded by the convex hull of its Bernstein coefficients, as shown in Figure~\ref{fig:hull_examples}.
Since $p(0) = \beta_0$ and $p(1) = \beta_{n_p}$, \emph{tight} bounds are available when the convex hull of $\{\beta_j\}_{j=0}^{n_p}$ coincides with the convex hull of $\{p(0), p(1)\}$. Figure~\ref{fig:hull_examples} shows an example of a tightly bounded polynomial and an example of a non-tightly bounded polynomial.

\subsection{Monotonicity}

\begin{figure}
    \centering
    \includegraphics[width=\linewidth]{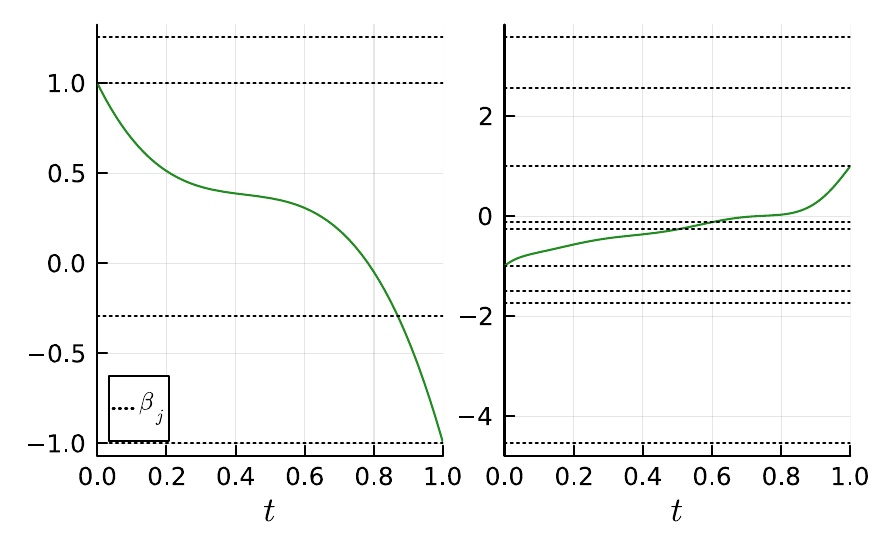}
    \vspace*{-9mm}\caption{Examples of monotonic polynomials that are not tightly bounded by their Bernstein coefficients. The polynomials are given in Appendix~A.}
    \label{fig:counter_examples}
\end{figure}

It may appear as if monotonic polynomials can always be tightly bounded.
In practice, such a result would help quantify the number of required sub-intervals for a tightly bounded approximation, should the number of monotonic segments of the target function be known.
Unfortunately, this is not true: Figure~\ref{fig:counter_examples} provides some examples.

Should the target function indeed have a finite number of monotonic segments, one may ask if a finite number of sub-intervals is sufficient to achieve a tightly bounded approximation. This result is formalized and proved in the following sub-section.

\subsection{Finite Number of Sub-Intervals}

The main result requires the following inequality.

\begin{lemma}
Let $\{c_k\}_{k=1}^{n_p}$ be a finite sequence of real numbers with $c_1 > 0$. There exists a scalar $0 < h < 1$ such that
\begin{equation}
    0 \leq c_1 h + \sum_{k=2}^{n_p} c_k h^k.
\label{eq:suff}
\end{equation}
\label{th:geometric}
\end{lemma}
\begin{proof}
    Let $c_{\min} := \min \{ c_k \}_{k=2}^{n_p}$ be used to provide the lower bound
    \begin{equation}
        0 \leq c_1 h + c_{\min} \sum_{k=2}^{n_p} h^k \leq c_1 h + \sum_{k=2}^{n_p} c_k h^k.
        \label{eq:lower_bound}
    \end{equation}

    The geometric series $S = h^2 + h^3 + ... + h^{n_p}$ can be expressed in closed form by subtracting $hS$ from $S$ and rearranging to obtain
    \begin{equation}
        S = \frac{h^2 - h^{n_p+1}}{1-h},
    \label{eq:geometric_series}
    \end{equation}
    for $0 < h < 1$. Substituting in the left inequality of \eqref{eq:lower_bound}, we have
    \begin{equation}
        0 \leq c_1 h + c_{\min} \frac{h^2 - h^{n_p+1}}{1-h},
    \end{equation}
    which can be simplified to
    \begin{equation}
        \frac{-c_{\min}}{c_1} \leq \frac{1-h}{h - h^{n_p}}.
    \end{equation}
    Introducing the lower bound
    \begin{equation}
        \frac{-c_{\min}}{c_1} \leq \frac{1-h}{h} \leq \frac{1-h}{h - h^{n_p}}
    \end{equation}
    shows that there always exists a sufficiently small $h$ for which $\frac{1-h}{h}$ is larger than any $\frac{-c_{\min}}{c_1}$ (since $\lim_{h \to 0^+} \frac{1 - h}{h} = +\infty$), thus proving the lemma.
\end{proof}


\begin{theorem}
    Let $p$ be a univariate polynomial of degree at most $n_p$  that is monotonic on a finite interval $[t_a, t_b]$. There exists a finite partition of $[t_a, t_b]$, such that~$p$ is tightly bounded by its Bernstein coefficients in every sub-interval of the partition.
\label{th:tight}
\end{theorem}

\begin{proof}
    Without loss of generality, we assume $p$ to be defined on the interval $[0, 1]$, and aim to find a sub-interval $[0, h]$, for some $0 < h < 1$, on which $p$ is tightly bounded by its Bernstein coefficients. 
    
    Let the monomial coefficients of $p$ be $\{\alpha_k\}_{k=0}^{n_p}$ and let $p_h$ be the following linear transformation of $p$,
    \begin{equation}
        p_h(t_h) := p(h t_h) \quad \forall t_h \in [0, 1].
    \end{equation}
    The monomial coefficients of $p_h$ are $\{ \alpha_k h^k\}_{k=0}^{n_p}$, and let $\{\beta_j\}_{j=0}^{n_p}$ be the Bernstein coefficients of $p_h$.
    
    For the case where $p$ is non-decreasing in $[0, 1]$, $p_h(0) = \beta_0$ and $p_h(1) = \beta_{n_p}$, and the conditions for tight bounds \eqref{eq:min_increasing} and \eqref{eq:max_increasing} require that
    \begin{equation}
        \beta_0 \leq \beta_j \leq \beta_{n_p},
    \end{equation}
    for $j \in \{1,2,...,n_p-1\}$. Equivalently, performing the change of basis \eqref{eq:bernstein_coefficients} to the monomial coefficients,
    \begin{equation}
        \alpha_0
        \leq
        \sum_{k=0}^{n_p} B_{j,k} \alpha_k h^k
        \leq
        \sum_{k=0}^{n_p} \alpha_k h^k.
    \end{equation}
    Since $B_{j,0} = 1$ $\forall j \in \{1,2,...,n_p-1\}$, $\alpha_0$ can be subtracted, leaving
    \begin{equation}
        0
        \leq
        \sum_{k=1}^{n_p} B_{j,k} \alpha_k h^k
        \leq
        \sum_{k=1}^{n_p} \alpha_k h^k.
        \label{eq:condition}
    \end{equation}
    In the case that $\alpha_k = 0$ $\forall k \in \{1,2,...,n_p\}$, \eqref{eq:condition} is trivially satisfied. Otherwise, for a sufficiently small $h$:
    \begin{itemize}
        \item The left inequality holds, given Lemma~\ref{th:geometric} and since the first non-zero element of $\{\alpha_k\}_{k=1}^{n_p}$ is positive for a non-decreasing, non-constant polynomial;
        \item The right inequality holds since all the elements of $B_{j,k}$, for $j \in \{1, 2, ...,n_p-1\}, k \in \{1, 2,..., n_p\}$, are strictly less than 1.
    \end{itemize}
    A similar argument can be made for the case where $p$ is non-increasing on $[t_a, t_b]$, using the tight bound conditions~\eqref{eq:max_decreasing} and \eqref{eq:min_decreasing}.
\end{proof}
    
In practice, Theorem~\ref{th:tight} suggests that a sufficient number of flexible sub-intervals must be chosen so as to tightly bound approximating polynomials. The monotonicity assumption is rather mild, since most functions can be approximated by piecewise-monotonic piecewise polynomials. Moreover, it has been shown that monotonic polynomials can approximate monotonic functions with a similar rate of convergence as in the unconstrained case \cite{de_vore_monotone_1977}.

\section{Dynamic Optimization Problems}
\label{sec:problem}

We define a DOP as finding the $n_x$ continuous states $\bs x : [t_0, t_f] \rightarrow \mb R^{n_x}$ and 
$n_u$ inputs $\bs u : [t_0, t_f] \rightarrow \mb R^{n_u}$ that
\begin{align*}
\def\arraystretch{1.2}
\begin{array}{rll}
    \text{minimize} & \multicolumn{2}{l}{m(\bs x(t_0), \bs x(t_f)) 
    + \hspace{-0.35em}
    \ds{\int_{t_0}^{t_f} \hspace{-0.35em}
    \ell(\bs x(t), \bs u(t), t) \textrm{d}t,}}\\ 
    \text{subject to} & b(\bs x(t_0), \bs x(t_f)) = 0,\\
    & r(\dot{\bs x}(t), \bs x(t), \bs u(t), t) = 0 & \forall t \in [t_0, t_f] \, \text{a.e.},\\
    & \bs x(t) \in \mb{X}, \, \bs u(t) \in \mb{U}  & \forall t \in [t_0, t_f].
\end{array}
\tag{P}
\label{eq:dop}
\end{align*}

The cost function combines a boundary cost $m : \mb R^{n_x} \times \mb R^{n_x} \rightarrow \mb R$ with a time-integrated cost $\ell : \mb R^{n_x} \times \mb R^{n_u} \times [t_0, t_f] \rightarrow \mb R$.
As equality constraints, $b : \mb R^{n_x} \times \mb R^{n_x} \rightarrow \mb R^{n_b}$ encompasses the $n_b$ boundary conditions and $r : \mb R^{n_x} \times \mb R^{n_x} \times \mb R^{n_u} \times [t_0, t_f] \rightarrow \mb R^{n_r}$ comprises the $n_r$ dynamic equations.
The latter are enforced \emph{almost everywhere} (a.e.) in the Lebesgue sense, due to the partition of $[t_0, t_f]$ into sub-intervals, as described in the next sub-section.
The sets $\mb X$ and $\mb U$ consist of simple inequality constraints for the states and inputs
\begin{align}
    \mb X &:= \{z \in \mb R^{n_x} \, | \, x_{\ell} \leq z \leq x_u\},\\
    \mb U &:= \{z \in \mb R^{n_u} \, | \, u_{\ell} \leq z \leq u_u\},
\end{align}
where $x_\ell, x_u \in \mb R^{n_x}$ and $u_\ell, u_u \in \mb R^{n_u}$ are the respective lower and upper constraint values.
General inequality constraints of the form
\begin{equation}
    g_\ell \leq g(\dot{\bs x}(t), \bs x(t), \bs u(t), t) \leq g_u
    \quad
    \wt{\forall} t \in [t_0, t_f]
\end{equation}
can be included in \eqref{eq:dop} by incorporating
\begin{equation}
    g(\dot{\bs x}(t), \bs x(t), \bs u(t), t) - \bs s(t) = 0
    \quad
    \wt{\forall} t \in [t_0, t_f]
\end{equation}
in the dynamic equations, with the slack input variables $\bs s$ constrained by $g_\ell$ and $g_u$. 

The method presented in this article can be easily extended to broader classes of DOPs, such as those with variable $t_0$ and~$t_f$, system parameters, and multiple phases.

\section{Flexible Discretizations}
\label{sec:discretization}

\subsection{Flexible Sub-Intervals}

Typically, the interval $[t_0, t_f]$ is partitioned into $n_h$ \emph{fixed} sub-intervals, defined by the values
\begin{equation}
    t_0 < t_1 < t_2 < ... < t_{n_h-1} < t_f.
\label{eq:rigid}
\end{equation}
Alternatively, $[t_0, t_f]$ can be partitioned into $n_h$ \emph{flexible} sub-intervals, defined by
\begin{equation}
    t_0 < \da t_1 < \da t_2 < ... < \da t_{n_h-1} < t_f,
\end{equation}
where $\{ \da t_i \}_{i=1}^{n_h-1}$ are optimization variables. It is often useful to restrict the sizes of the flexible sub-intervals.
We define the flexibility parameters $\{ \phi_i \}_{i=1}^{n_h}$, each within $[0, 1)$, to include the inequality constraints
\begin{multline}
    (1 - \phi_i)(t_i - t_{i-1})
    \leq 
    \da t_i - \da t_{i-1} 
    \leq
    \phi_i (t_f - t_0) +\\ (1 - \phi_i)(t_i - t_{i-1})
\end{multline}
for each sub-interval $i \in \{1, 2, ..., n_h\}$, where $\da{t}_0 \, \equiv t_0$ and $\da{t}_{n_h} \, \equiv t_f$.
In the special case where all flexibility parameters $\{ \phi_i \}_{i=1}^{n_h}$ equal~0\%, the fixed partitioning \eqref{eq:rigid} is recovered.

\subsection{Dynamic Variables}
For the $i^\text{th}$ (flexible) interval, the states and inputs are discretized by interpolating polynomials. To simplify notation, the interpolations are defined in the  normalized time domain $\tau \in [-1, 1]$, mapped to $t \in [\da t_{i-1}, \da t_i]$ via
\begin{equation}
    \gamma(\tau; \da t_{i-1}, \da t_i) :=
    \frac{\da t_i - \da t_{i-1}}{2} \tau +
    \frac{\da t_{i-1} + \da t_i}{2}.
\end{equation}
Using a Legendre-Gauss-Radau (LGR) collocation method~\cite{kameswaran_convergence_2008} of degree $n$, the interpolation points $\{ \tau_j \}_{j=0}^n$ are the $n$ LGR collocation points, with the extra point $\tau_n = 1$.

\begin{figure}
    \centering
    \includegraphics[width=\linewidth]{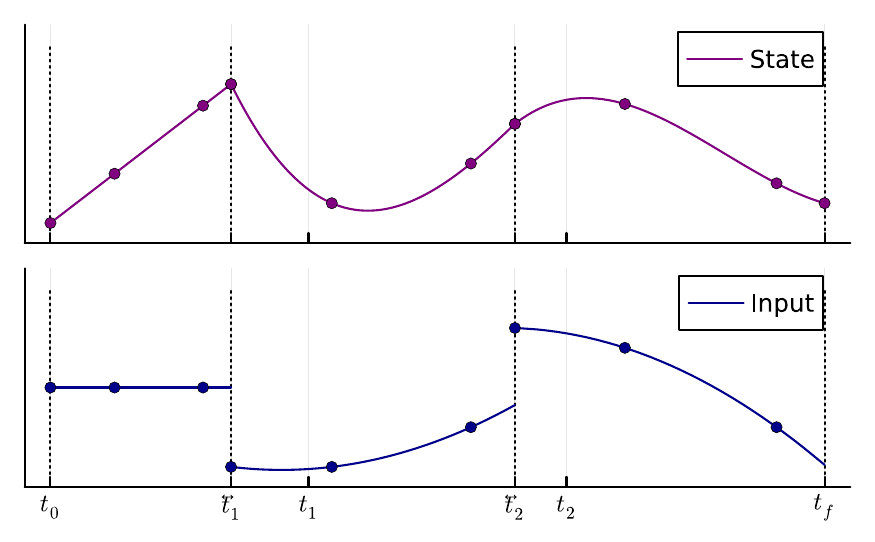}
    \vspace*{-9mm}\caption{Representations of flexible state and input discretizations for an LGR collocation method of degree $n = 3$.}
    \label{fig:dynamic_variables}
\end{figure}

The function $\wt x_n : \mb [{-}1, 1] \times \mb R^{(n+1) \times n_x} \rightarrow \mb R^{n_x}$ approximates the states with $n_x$ Lagrange interpolating polynomials of degree $n$ as
\begin{equation}
    \wt x_n(\tau; x_i) := \sum_{j=0}^{n}
    \Bigg( x_{i,j}
    \prod_{\substack{k = 0\\ k \neq j}}^{n} \frac{\tau - \tau_k}{\tau_j - \tau_k}
    \Bigg),
\end{equation}
where $x_{i,j}$ are the states at the $j^{\text{th}}$ interpolation point of the~$i^{\text{th}}$ sub-interval and $x_i$ is $\{ x_{i,j} \}_{j=0}^n$. The continuity between sub-intervals is preserved by enforcing 
\begin{equation}
    \wt x_n(1, x_i) = \wt x_n(-1; x_{i-1}),
\end{equation}
for $i \in \{2,..., n_h\}$.
The function $\wt u_n : \mb [{-}1, 1] \times \mb R^{n \times n_u} \rightarrow \mb R^{n_u}$ approximates the inputs with $n_u$ Lagrange interpolating polynomials of degree $n-1$ as
\begin{equation}
    \wt u_n(\tau; x_i) := \sum_{j=0}^{n-1}
    \Bigg( u_{i,j}
    \prod_{\substack{k = 0\\ k \neq j}}^{n-1} \frac{\tau - \tau_k}{\tau_j - \tau_k}
    \Bigg),
\end{equation}
where $u_{i,j}$ are the inputs at the $j^{\text{th}}$ collocation point of the~$i^{\text{th}}$ sub-interval and $u_i$ is $\{ u_{i,j} \}_{j=0}^{n-1}$.
Figure~\ref{fig:dynamic_variables} illustrates the discretization of the dynamic variables in flexible sub-intervals.

Other variants of pseudo-spectral methods can also be used, such as those based on Chebyshev polynomials. Additionally, pseudo-spectral methods based on Birkhoff interpolation \cite{koeppen_fast_2019} are also compatible with the proposed framework.

\subsection{Cost Function}
The boundary cost of the discretized states is simply
\begin{equation}
    m(\wt x_n(-1; x_0), \wt x_n(1; x_{n_h})).
\end{equation}
The time-integrated cost is numerically approximated by Gaussian quadrature as
\begin{multline}
    \sum_{i=1}^{n_h} \frac{\da t_i - \da t_{i-1}}{2}
    \sum_{j=0}^{n-1} w_j \ell \Big(
    \wt x_n(\tau_j; x_i),\\
    \wt u_n(\tau_j; u_i),
    \gamma \Big( \tau_j, \da t_{i-1}, \da t_i\Big)
    \Big),
\end{multline}
where $\{ \tau_j \}_{j=0}^{n-1}$ and $\{ w_j \}_{j=0}^{n-1}$ are the $n$ LGR quadrature points and weights, respectively.

\subsection{Equality Constraints}

The boundary conditions of the discretized states are simply enforced by
\begin{equation}
    b(\wt x_n(-1; x_0), \wt x_n(1; x_{n_h})) = 0.
\end{equation}
The dynamic equations are enforced at every collocation point $\{ \tau_j\}_{j=0}^{n-1}$ by setting
\begin{multline}
    r\Bigg(
    \frac{2\dot{\wt x}_n(\tau_j; x_i)}{\da t_i - \da t_{i-1}},
    \wt x_n(\tau_j, x_i),
    \wt u_n(\tau_j; u_i),\\
    \gamma(\tau_j, \da t_{i-1}, \da t_i)
    \Bigg) = 0,
\end{multline}
for every sub-interval $i \in \{1, 2,..., n_h\}$.

\subsection{Inequality Constraints}

In pseudo-spectral methods, the inequality constraints in~\eqref{eq:dop} are often only enforced at the interpolation points, i.e.
\begin{equation}
    x_{i,j} \in \mb X, \quad u_{i,j} \in \mb U.
\label{eq:lazy_bounds}
\end{equation}
This does not necessarily imply that
\begin{equation}
    \wt x_n(\tau, x_i) \in \mb X, \quad \wt u_n(\tau, u_i) \in \mb U, \quad \forall \tau \in [-1, 1]
\label{eq:safe_bounds}
\end{equation}
is satisfied for all sub-intervals $i \in \{1, 2,..., n_h\}$.
On the other hand, constraining the Bernstein coefficients of $\wt x$ and~$\wt u$ ensures that \eqref{eq:safe_bounds} is satisfied.

\subsection{Bounds on Interpolating Polynomials}
Let $\wt y : [{-}1, 1] \times \mb R^{n_p+1} \rightarrow \mb R$, $ (\tau, y) \mapsto \wt y(\tau; y)$ be the Lagrange interpolating polynomial at points $\{\tau^y_j\}_{j=0}^{n_p}$ with coefficients $y$. The Bernstein coefficients of $\wt y(\cdot; y)$ can be obtained by
\begin{equation}
    \beta = C y,
    \quad C = B V^{-1},
\label{eq:lagrange-bernstein}
\end{equation}
where $V$ is the Vandermonde matrix for the linearly scaled points $\{0.5 \tau^y_j + 0.5\}_{j=0}^{n_p}$.
Hence, for the scalars $y_\ell < y_u$, 
\begin{equation}
    y_\ell \leq C y \leq y_u
    \implies
    y_\ell \leq \wt y(\tau; y) \leq y_u,
    \,
    \forall \tau \in [-1, 1].
\label{eq:good_bounds}
\end{equation}

This approach is used to enforce the inequality constraints of \eqref{eq:dop} on the discretized dynamic variables, thus satisfying
\begin{equation}
    \wt x(\tau; x_i) \in \mb X,
    \quad
    \wt u(\tau; u_u) \in \mb U,
    \quad
    \forall \tau \in [-1, 1],
\end{equation}
for every sub-interval $i \in \{1,2,...,n_h\}$.

\subsection{Discretized Problem}

The resulting discretized optimization problem is
\begin{align*}
    \begin{array}{rlr}
        \ds{\min_{x, u, \da t}}
        &\multicolumn{2}{l}{
            m \big( \wt x_n(-1; x_0), \wt x_n(1; x_{n_h}) \big) +
            \ds{\sum_{i=1}^{n_h}} \Bigg[ \frac{\da t_i - \da t_{i-1}}{2}}\\
        &\multicolumn{2}{r}{
            \ds{\sum_{j=0}^{n-1}} w_j \ell \big(
            \wt x_n(\tau_j; x_i),
            \wt u_n(\tau_j; u_i),
            \gamma (\tau_j, \da t_{i-1}, \da t_i)
            \big) \Bigg],
        }\\
        \text{s.t.}
        &b \big(
            \wt{x}(-1, x_0),
            \wt{x}(1; x_{n_h})
        \big) = 0,\\
        &\multicolumn{2}{l}{
            r \Big(
                \frac{2\dot{\wt x}_n(\tau_j; x_i)}{\da t_i - \da t_{i-1}},
                \wt x_n(\tau_j, x_i),
                \wt u_n(\tau_j; u_i),}\\
        &\multicolumn{2}{r}{
                \gamma(\tau_j, \da t_{i-1}, \da t_i)
            \Big) = 0, \quad \forall j \in \{1, 2, ..., n\},
        }\\
        &C_{n+1} x_{i,:,k_x} \in \mb{X}_{k_x}, \quad
        &\forall k_x \in \{1, 2, ..., n_x\},\\
        &C_{n} u_{i,:,k_u} \in \mb{U}_{k_u}, 
        &\forall k_u \in \{1, 2, ..., n_u\},\\
        &\da{t}_{i+1} - \da{t}_i \in \mb{T}_{\phi}
        &\forall i \in \{1, 2, ..., n_h\},
    \end{array}
    \tag{$\text{P}_n$} 
\end{align*}
where $C_n$ is the transformation matrix from interpolation coefficients to Bernstein coefficients, as in \eqref{eq:lagrange-bernstein}, for a polynomial of degree $n$.  


\section{Examples}
\label{sec:example}
The proposed method is demonstrated on two example DOPs: the Bryson-Denham problem, and a constrained cart-pole swing-up problem. Numerical results were obtained using Ipopt~\cite{wachter_implementation_2006} and JuMP~\cite{lubin_jump_2023} software packages with double-precision floating-point representations.

\subsection{Example DOPs}
The Bryson-Denham problem \cite[Sect. 3.11]{bryson_applied_1975} consists of a single input, with double-integrator dynamics on a state $\bs r$, and a parametrizable inequality constraint, here chosen as
\begin{equation}
    \bs{r}(t) \leq 0.2 \quad \forall t \in [0, 1].
\end{equation}
This problem was selected given that an analytical solution exists\cite[Sect. 3.11]{bryson_applied_1975}, allowing for a complete assessment of approximate solutions.

The cart-pole swing-up problem~\cite[Sect.~6]{kelly_introduction_2017}~\cite[App.~E.1]{kelly_introduction_2017} consists of a single input, for the force acting on the cart, and four states $\bs q_1, \bs q_2, \dot{\bs q}_1, \dot{\bs q}_2$, where $\bs q_1$ is the position of the cart and $\bs q_2$ is the angle of the pole. Because none of the inequality constraints are active at the solution to the original problem, we consider the case where the cart's position is further constrained to satisfy
\begin{equation}
    0 \leq \bs q_1(t) \leq 1, \quad \forall t \in [t_0, t_f].
\end{equation}

\subsection{Approximate Solutions}

For both problems, the following discretization approaches have been used to obtain approximate solutions:
\begin{enumerate}[label=(\alph*)]
    \item Equispaced sub-intervals, with inequality constraints enforced at the interpolation points, as per \eqref{eq:lazy_bounds};
    \item Equispaced sub-intervals, with inequality constraints enforced on the Bernstein coefficients, as per \eqref{eq:good_bounds};
    \item Flexible sub-intervals, with inequality constraints enforced on the Bernstein coefficients, as per \eqref{eq:good_bounds}.
\end{enumerate}

\begin{figure}
    \centering
    \includegraphics[width=\linewidth]{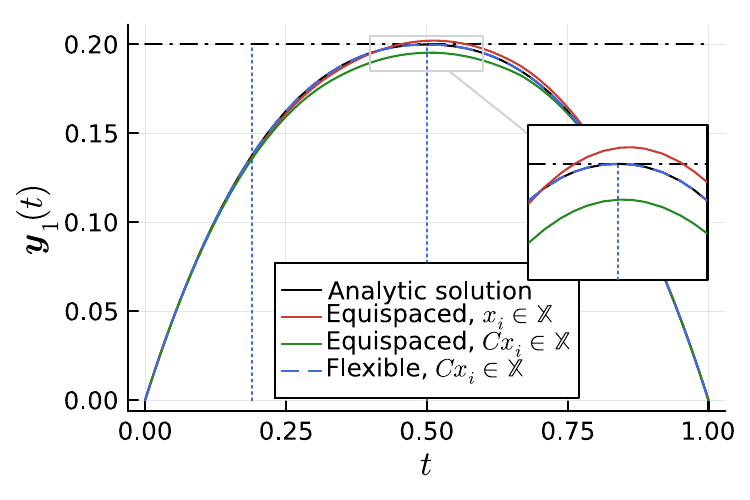}
    \vspace*{-9mm}\caption{Trajectories of $\bs r$ for the three discretizations, with dashed lines indicating the flexible sub-intervals.}
    \label{fig:bryson_denham_solutions}
\end{figure}

\begin{figure}
    \centering
    \includegraphics[width=\linewidth]{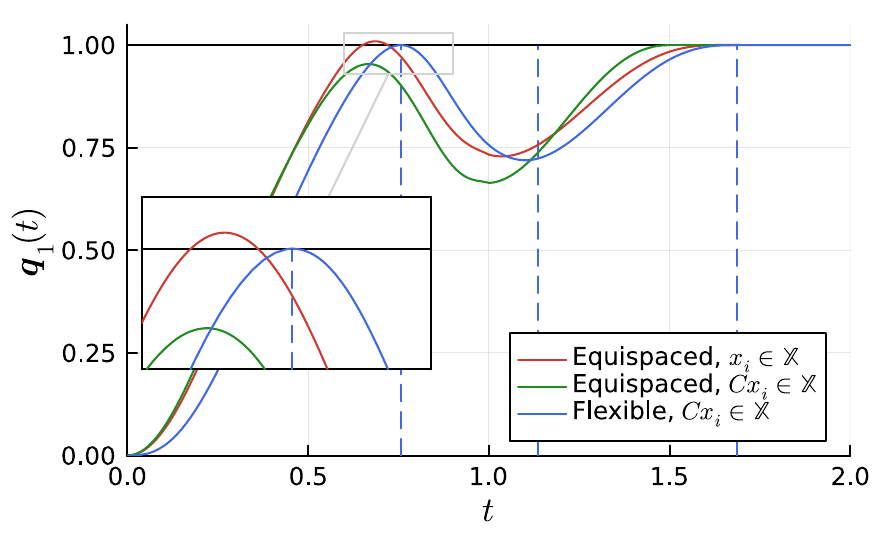}
    \vspace*{-9mm}\caption{Trajectories of the cart's position for the three discretizations, with dashed lines indicating the flexible sub-intervals.}
    \label{fig:cartpole_solutions}
\end{figure}

Figure~\ref{fig:bryson_denham_solutions} shows approximate solutions of $\bs r$, using the three approaches with LGR collocation degree $n = 3$, $n_h = 3$ sub-intervals, and $\phi = 50\%$ flexibility. Figure~\ref{fig:cartpole_solutions} shows approximate solutions of $\bs q_1$ with LGR collocation degree $n = 8$, $n_h = 4$ sub-intervals, and $\phi = 50\%$ flexibility. 

In both cases, approach (a) violates the inequality constraint, whereas there are no violations with approaches (b) and (c), as expected.
It can also be seen how approach (b) is rather conservative, in contrast to approach (c). In these examples, as with many others, there is a cost incentive to operate near the constraints.

\subsection{Assessment Criteria}
To assess an approximate solution $(x^*, u^*)$, the following criteria are considered. 

\subsubsection{Cost}
The cost expressions for the Bryson-Denham problem, and the cart-pole swing-up problems are respectively
\begin{equation}
    \int_{t_0}^{t_f} \frac{1}{2}u^*(t)^2 \textrm{d}t \quad \text{and} \quad 
    \int_{t_0}^{t_f} u^*(t)^2 \textrm{d}t.
\label{eq:cost_integration}
\end{equation}

\subsubsection{Inequality Constraint Violation}

We define the violation function for a scalar-valued trajectory $y : [t_0, t_f] \rightarrow \mathbb R$ with lower and upper constraints $y_\ell$ and $y_u$ respectively, as
\begin{equation}
v(t; y, y_\ell, y_u) :=
    \begin{cases}
        y_\ell - y(t),  & \text{if $y(t) < y_\ell$}\\
        y(t) - y_u,     & \text{if $y(t) > y_u$}\\
        0,              & \text{otherwise.}
    \end{cases}
\end{equation}
The \emph{total} inequality constraint violation is defined as the sum of violation norms for the input and the states, i.e. 
\begin{equation}
    \| v(\cdot; u^*, u_\ell, u_u) \|_2
    +
    \sum_{k = 1}^{n_x} \| v(\cdot; x^*_k, x_{\ell, k}, x_{u, k}) \|_2,
\label{eq:inequality_integration}
\end{equation}
where the $L_2$-norm $\| f \|_2 := \sqrt{\int_{t_0}^{t_f} |f(t)|^2 \textrm{d}t}$, for $f : [t_0, t_f] \rightarrow \mb R$.
The input constraints are given by $u_\ell$ and $u_u$, and the $k^\text{th}$ state constraints are given by $x_{\ell, k}$ and $x_{u,k}$.

\subsubsection{Dynamic Constraint Violation}

This is defined as the average $L_2$-norm of the violation of the four dynamic equations, i.e.
\begin{equation}
    \frac{1}{n_x}
    \sum_{k=1}^{n_x} 
    \| r_k(\dot{x}^*(\cdot), x^*(\cdot), u^*(\cdot), \cdot) \|_2.
\label{eq:dynamics_integration}
\end{equation}

\subsection{Convergence}

\begin{figure}
    \centering
    \includegraphics[width=\linewidth]{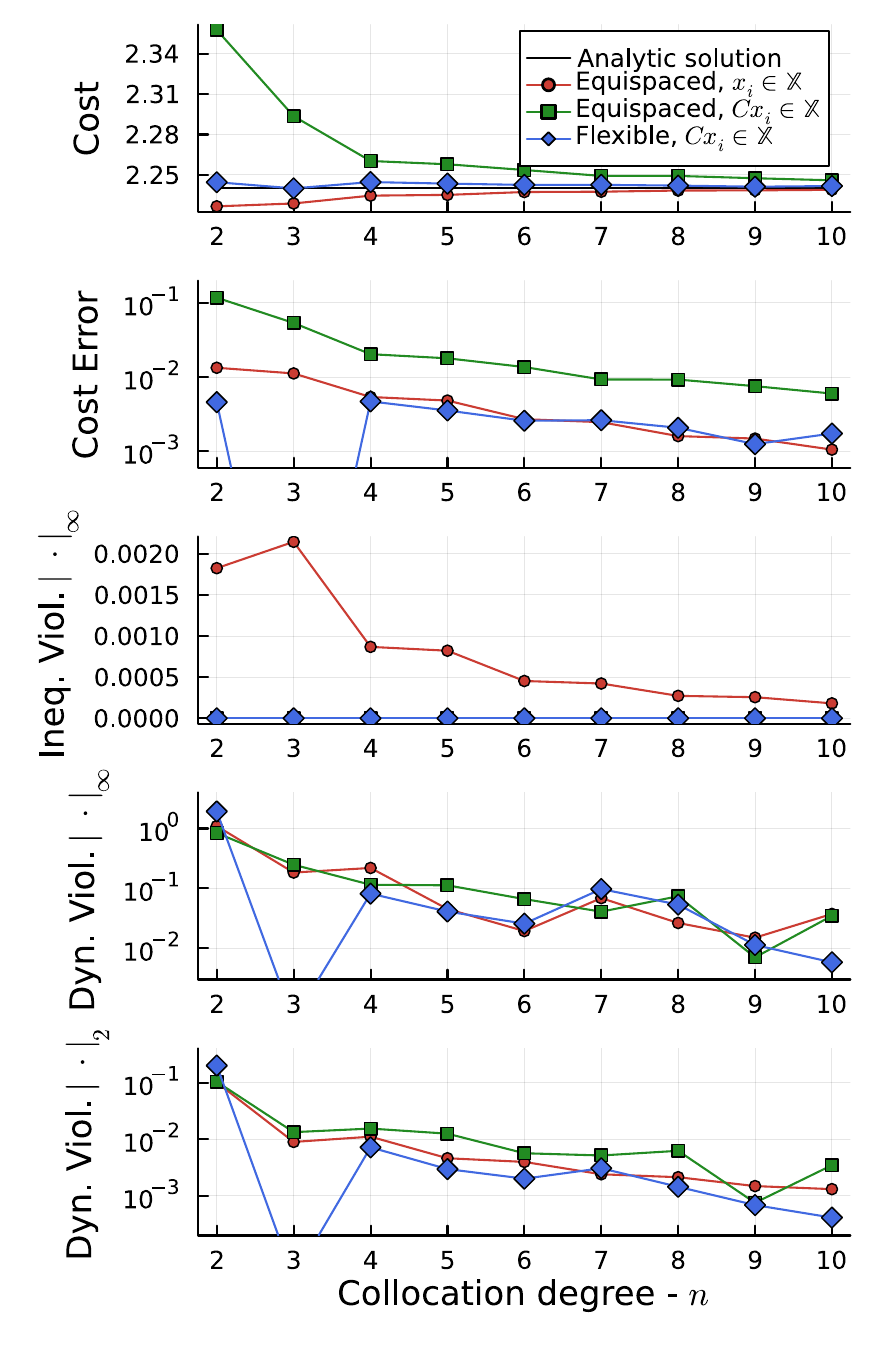}
    \vspace*{-9mm}\caption{Assessment of numerical solutions to the Bryson-Denham problem, with dynamic constraint violation on log scales.}
    \label{fig:bryson_denham_convergence}
\end{figure}

\begin{figure}
    \centering
    \includegraphics[width=\linewidth]{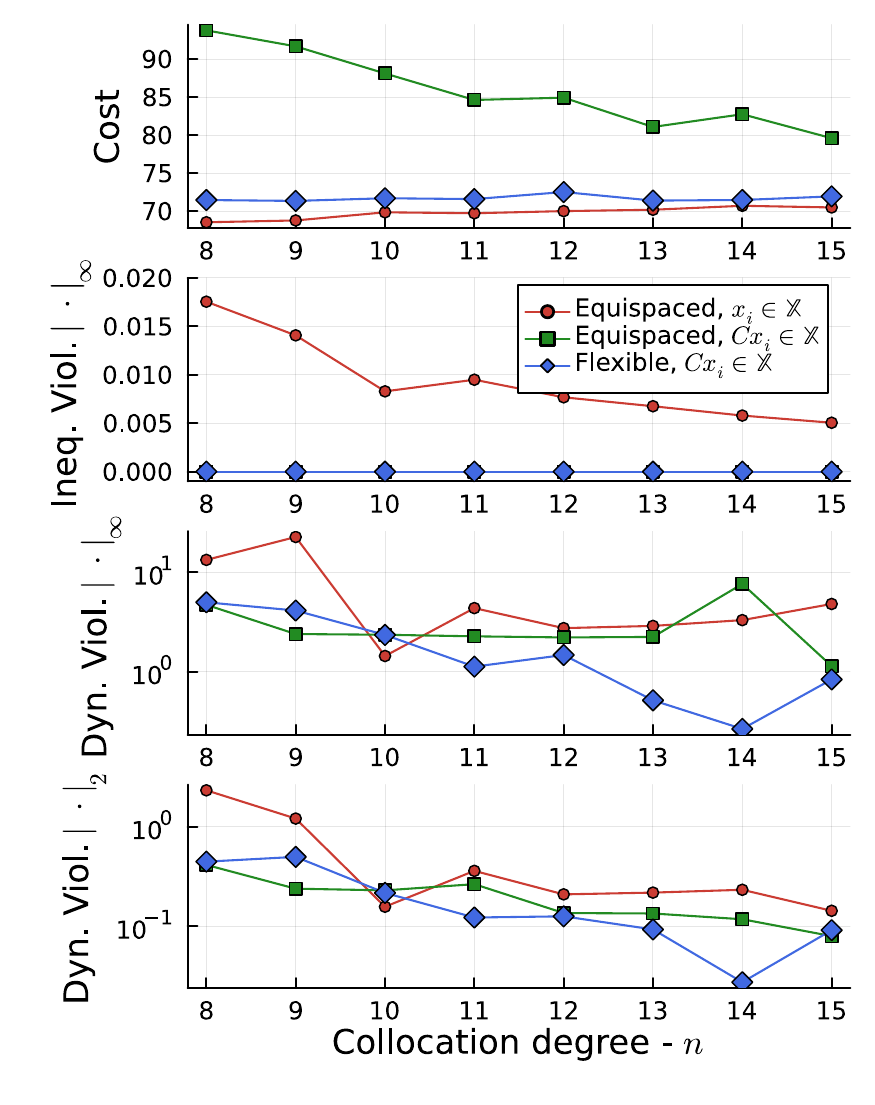}
    \vspace*{-9mm}\caption{Assessment of numerical solutions to the cart-pole swing-up problem, with dynamic constraint violation on log scales.}
    \label{fig:cartpole_convergence}
\end{figure}

Figure~\ref{fig:cartpole_convergence} shows a comparison of approaches (a), (b) and~(c), for each of the three criteria.
Even with higher polynomial degrees, approach (a) continues to violate the inequality constraints.
With this violation, approach (a) is able to obtain a smaller cost, in comparison with (b) and~(c).
Approach (b) reports a significantly higher cost, due to the conservative Bernstein bounds. The sub-interval flexibility allows approach (c) to eliminate the conservatism of the Bernstein bounds and obtain a smaller cost, in comparison with approach (b).

It should be noted, however, that for lower $n$, the solutions exhibit a higher dynamic constraint violation, which, in practice, may demerit the large differences in cost.  

\subsection{Sub-Interval Flexibility}

\begin{figure}
    \centering
    \includegraphics[width=\linewidth]{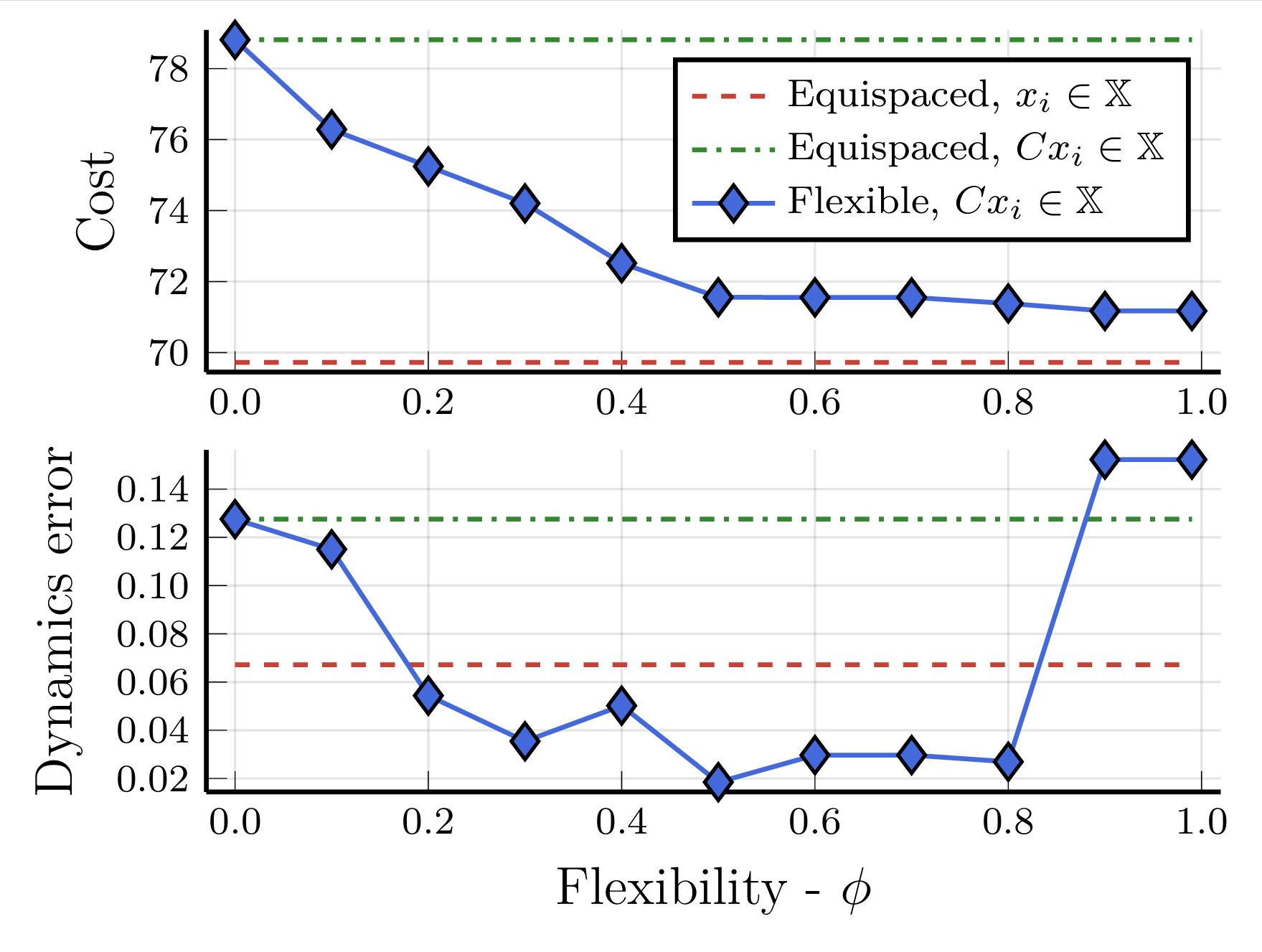}
    \vspace*{-9mm}\caption{Cost and dynamic constraint violation for a range of flexibility parameters for the cart-pole swing-up problem. Dashed lines compare the solutions of approaches (a) and (b).}
    \label{fig:cartpole_flexibility}
\end{figure}

Figure~\ref{fig:cartpole_flexibility} shows the impact of the flexibility parameter on the cost and dynamic error.
As expected, in the special case of $\phi = 0\%$, approach (c) is indistinguishable from approach (b).
For $n = 7$, a flexibility of $\phi = 50\%$ allows for a significant decrease in cost, without compromising on the dynamic constraint violation. Ample flexibility may create large sub-intervals, resulting in locally less dense discretizations, thus increasing the overall dynamic constraint violation of the solution.


\section{Conclusion}
\label{sec:conclusion}

A pseudo-spectral method for DOPs was presented that not only rigorously enforces inequality constraints but also allows for tight polynomial bounds to be achieved via flexible sub-intervals.
A proof was provided, demonstrating the ability of a piecewise-monotone polynomial to be tightly constrained using a finite set of sub-intervals.

Flexible sub-interval discretizations may also result in more challenging optimization problems. This is not only due to the introduction of non-linearities in the dynamic equations but also due to non-unique solutions, given the different combinations of equally optimal sub-interval lengths.

Further research could explore convexification and regularization techniques, as well as conditions under which tightly-constrained trajectories correspond to stationary points of the optimization problems.

\appendices

\section{Example Polynomial Coefficients}
\label{ap:coefficients}
The polynomials in Figure~\ref{fig:counter_examples} are defined by Lagrange interpolation of the points {1.0, 0.4, -0.2, -1.0}, and { -1.0, -0.8, -0.6, -0.4, -0.2, 0.0, 0.2, 0.8, 1.0}, at the LGR nodes.

\section{Sine Approximation Problem}

The function approximation problem in Section~\ref{sec:introduction} is defined as finding $\bs y : [0, 1] \rightarrow \mb R$ that minimizes
\begin{equation}
    \int_0^1(\sin(2 \pi t) - \bs y(t))^2 \textrm{d}t,
\label{eq:sine_integral}
\end{equation}
subject to the inequality constraints $-1 \leq \bs y(t) \leq 1, \forall t \in [0, 1]$. In the case of flexible sub-intervals, \eqref{eq:sine_integral} is approximated by 
\begin{equation}
    \sum_{i=1}^{3}
    \frac{\da{t}_i - \da{t}_{i-1}}{2}
    \sum_{q=0}^{n_q} w_q 
    \big( \sin(2 \pi \gamma(\tau_q, \da t_{i-1}, \da t_i)) - \wt y(\tau_q; y_i)
    \big)^2,
\end{equation}
where $\{\tau_q\}_{q=0}^{n_q}$ and $\{w_q\}_{q=0}^{n_q}$ are the Legendre-Gauss-Lobatto (LGL) quadrature points and weights, respectively. And $\wt y$ is an interpolating polynomial using $n_p + 1$ LGL points. It is chosen that $n_q = n_p + 2$.

\section{Numerical Integration}

The numerical integration in \eqref{eq:cost_integration}, \eqref{eq:inequality_integration}, and~\eqref{eq:dynamics_integration} was performed using the software package QuadGK.jl (version 2.9.4, available from \texttt{https://github.com/JuliaMath/QuadGK.jl}) with the default options.

\bibliographystyle{ieeetr}
\bibliography{references}

\end{document}